\documentclass[12pt,reqno]{amsart}
\usepackage{graphicx,amssymb}
\input xy
\xyoption{all}
\newtheorem*{thm}{Theorem}
\newtheorem*{cor}{Corollary}
\newtheorem*{lem}{Lemma}
\newtheorem*{prop}{Proposition}
\newtheorem*{claim}{Claim}
\theoremstyle{definition}
\newtheorem*{defn}{Definition}
\theoremstyle{remark}
\newtheorem*{rem}{Remark}

\newcommand{\Shadow}{\operatorname{Shadow}}

\newcommand{\Aut}{\operatorname{Aut}}

\newcommand{\diam}{\operatorname{diam}}
\newcommand{\defeq}{\stackrel{\text{def}}{=}}

\newcommand{\Conv}{\operatorname{Conv}}
\newcommand{\level}{\operatorname{level}}
\newcommand{\Star}{\operatorname{Star}}
\newcommand{\Sym}{\operatorname{Sym}}

\newcommand{\Haar}{\operatorname{Haar}}
  
\newcommand{\one}{{\mathbf{O}}}
\newcommand{\id}{{\mathit{id}}}

\newcommand{\arrow}{\longrightarrow}
\newcommand{\R}{\mathbf R}

\newcommand{\N}{\mathbf N}
\newcommand{\Z}{\mathbf Z}

\newcommand{\T}{\mathcal{T}}

\newcommand{\x}{{\bf{x}}}
\newcommand{\oxi}{\overline{\xi}}
\newcommand{\OXi}{\overline{\Xi}}

\begin{document}
\title{Most actions on regular trees are almost free}
\author{Mikl\'{o}s Ab\'{e}rt and Yair Glasner}
\address{University of Chicago, and Ben Gurion university of the Negev}
\email{yairgl@math.bgu.ac.il and abert@math.uchicago.edu}
\maketitle

\begin{abstract}
Let $T$ be a $d$-regular tree ($d\geq 3$) and $A=\mathrm{Aut}(T)$ its
automorphism group. Let $\Gamma $ be the group generated by $n$ independent
Haar-random elements of $A$. We show that almost surely, every nontrivial
element of $\Gamma $ has finitely many fixed points on $T$.
\end{abstract}

\section{Introduction}

Let $T$ be a $k$-regular tree ($k\geq 3$) and let $\mathrm{Aut}(T)$ be its
automorphism group. The topology of pointwise convergence turns $\mathrm{Aut}%
(T)$ into a locally compact, totally disconnected, unimodular topological
group. Let $\mu $ be a Haar measure on $\mathrm{Aut}(T)$ normalized so that
vertex stabilizers have measure $1$. Since $\mathrm{Aut}(T)$ is not compact, 
$\mu $ is an infinite measure.

\begin{defn}
A subgroup $\Gamma \leq \mathrm{Aut}(T)$ acts \emph{almost freely}, if every 
$\gamma \in \Gamma $ ($\gamma \neq 1$) has finitely many fixed points on $T$.
\end{defn}

Free actions on $T$ are completely understood (see Serre's book \cite{Serre:Trees}). As the following theorem and its corollary show, almost free actions have a much richer structure.

\begin{thm} (Main theorem)
Let $\Gamma <\mathrm{Aut}(T)$ be a countable subgroup acting almost freely.
Then for $\mu $-almost all elements $\gamma \in \mathrm{Aut}(T),$ the group $%
\left\langle \Gamma ,\gamma \right\rangle $ acts almost freely and is
isomorphic to the free product $\Gamma \ast \mathbb{Z}$.
\end{thm}

\begin{cor} \label{cor:main}
Let $a_{1},\ldots ,a_{n}$ be independent Haar-random elements of $\mathrm{Aut%
}(T)$ and let $\Gamma =\left\langle a_{1},\ldots ,a_{n}\right\rangle $. Then
almost surely, $\Gamma $ is a free group of rank $n$ that acts almost freely
on $T$.
\end{cor}

In a previous paper \cite{AG:Generic}, the authors proved that for $n\geq 2$, the closure
of $\Gamma $ almost surely satisfies the following trichotomy. It is either:

\begin{enumerate}
\item  \label{itm:disc} discrete in $\mathrm{Aut}(T)$;
\item \label{itm:cpt} or fixes a point or a geometric edge of $T$;
\item \label{itm:dense} or has index at most $2$ in $\mathrm{Aut}(T)$.
\end{enumerate}
\noindent
We also showed that all possibilities happen on a set of infinite Haar
measure. In another paper of the first author and Vir\'{a}g \cite{AV-dimension_theory}, it is proved
that random subgroups as above acting on a rooted tree act almost freely almost surely. This effectively proves Theorem 2 when case (\ref{itm:cpt}) of the trichotomy holds, but for the general result, a new approach is needed.

Like most random results, the Corollary to the main theorem can be used to show the existence of
structures that are hard to construct directly. In particular, the Corollary to the main theorem and part (\ref{itm:dense}) of the trichotomy result together imply that there \emph{exists} a finitely generated dense free subgroup of $\mathrm{Aut}(T)$ that acts almost freely on $T$.

\subsection*{Remark.} In proving the trichotomy above, an essential tool is to understand how $n$-tuples behave under the action of the so-called Nielsen transformations. In a forthcoming paper \cite{G:ZO}, the second author shows that this action is actually ergodic on the two compontents appearing in case (\ref{itm:dense}) of the dichotomy: the compoenent where $\Gamma$ is dense and the one where it is dense in a subgroup of index two. Since acting almost freely is a measurable property of $n$-tuples, to obtain the main Corollary in case (\ref{itm:dense}), it would be enough to show that it holds on a set of positive measure on both components. However, this does not seem to be any easier than proving the full statement.

\subsection*{Organization of the paper}
The paper is organized as follows. After introducing some notation and
preliminary results in Section 2, we prove the main Theorem in Section 3. The notation is such that when we refer to, say, Definition 3.3 we mean the definition in section 3.3 (in particular there will be only one such).

\subsection*{Thanks} The work on this paper was made possible by the support of a joint BSF grant 2006222, for which we are thankful.

\section{Notation and preliminaries} \label{sec:notation}
\subsection{Trees}
Let $T$ be a $d$-regular tree with automorphism group $A = \Aut(T)$. The set of vertices will also be denoted by $T$, the set of directed edges $ET$. Each edge $e \in ET$ admits an inverse edge $\overline{e}$ as well as an origin vertex $oe \in T$ and a terminal vertex $te \in T$. These maps are required to satisfy the obvious compatibility restrictions; namely $\overline{\overline{e}} = e$, $t(\overline{e}) = oe$ and $o(\overline{e}) = te$. {\it Geometric edges} will be represented by unordered pairs $[e]:=\{e,\overline{e}\}$. The {\it star of a vertex} $\Star(v) = \{e \in ET \ | oe = v\}$ is the set of edges originating at this vertex. Fix a base vertex $\one \in T$ and let $A_0 \defeq A_{\one} < A$ be the compact group fixing $\one$.  The standard graph metric on $T$ will be denoted by $d:T \times T \arrow T$. We will denote the unique geodesic path connecting two vertices $x,y \in T$ by $[x,y]$.

\subsection{The boundary of the tree}
A geodesic ray is an embedding of the half line $\eta: \R^{+} \arrow T$ into the tree. We say that two such rays $\eta,\eta'$ are equivalent if there exists a number $b$ such that $\eta(n) = \eta'(n + b) \ \ \forall n > N$ for some $N \in \N$. We define the boundary of the tree $\partial T$ to be the set of all equivalence classes of geodesic rays. There is a natural topology on $\overline{T} := T \cup \partial T$ making it into a compact space and the action of $\Aut(T)$ on $T$ extends in a continuous way to an action on this compactification. For every two points $x,y \in \overline{T}$ there is a unique geodesic path connecting these two points, which can be finite, half infinite, or bi-infinite depending on whether none, one or both are boundary points. 

\subsection{Dynamics of tree automorphisms} \label{sec:element_class}
We recall here briefly the classification of automorphisms of $T$ according to their dynamical properties, referring the readers to Serre's book \cite{Serre:Trees} for the complete details.
For every $\phi \in \Aut(T)$ we define $\delta(\phi) = min_{x \in T} \{d(x, \phi x)\}$ and $X(\phi) = \{x \in T \ | \ d(x,\phi x) = \delta(\phi)\}$. There are three possibilities that arise:
\begin{enumerate}
\item $\phi$ fixes a vertex of the tree. In this case $\delta (\phi) = 0$ and $X(\phi)$ is the tree of fixed points for $\phi$. Such a $\phi$ is called {\it{elliptic}}. 
\item $\phi$ does not fix a vertex, but it inverts a geometric edge. In this case $\delta(\phi) = 0$ and $X(\phi)$ is one point - the midpoint of the edge inverted by $\phi$. Such a $\phi$ is called {\it{an inversion}}.
\item $\phi$ does not fix any point in the geometric realization of the tree. Such an element is called {\it{hyperbolic}}. In this case $\phi$ fixes exactly two points on the boundary and $X(\phi)$ is the bi-infinite geodesic line connecting these two points,  referred to as the {\it{axis}} of $\phi$. A hyperbolic element acts as a translation of length $\delta(\phi)$ on its axis. For any other point $y \in T$ we have $d(y,\phi y) = \delta(\phi) + 2 d(y, X(\phi))$. 
\end{enumerate}

Using this classification, it is easy to see that a group $\Gamma < \Aut(T)$ acts almost freely on $T$ if and only if all elliptic elements of $\Gamma $ do not fix any point on the boundary of $T$. Indeed hyperbolic elements and inversions do not fix any vertex, and an elliptic element fixes infinitely many vertices if and only if it fixes a point on the boundary, by K\"{o}ning's lemma. Note that $\Gamma$ acts freely if and only if it contains only hyperbolic elements.

\subsection{Orientation} An orientation on $T$ is a choice of one directed edge $e$ from every pair of opposite directed edges $\{e,\overline{e}\}$. We will fix one orientation, namely that of all edges facing away from the base vertex, throughout this paper $$E^{+}T \defeq \{e \in ET \ | \ d(\one,te) > d(\one,oe) \} \subset ET.$$

\subsection{Legal coloring} Let $\Sigma = \{0,1,\ldots, d-1\}$, we will refer to this set as our {\it set of colors}. A legal coloring is a map $c: ET \arrow \Sigma$ such that for every vertex $v \in T$, the restricted map $$c|_{\Star(v)}: \Star(v) \arrow \Sigma$$ is a bijection. Note that we do not specify any requirement concerning the color of opposite edges. We will fix a legal coloring $c$ throughout this paper and require that it be compatible with the orientation in the sense that all the negative edges are colored zero$$ce = 0 \ \forall e \in ET \setminus E^+T.$$

\subsection{The first congruence map}
Let $\Sym(\Sigma)$ be the symmetric group on the set of colors and $\Sym(\Sigma)_0$ the stabilizer of the first color ``$0$''.
The legal coloring gives rise to a {\it first congruence map} associating to each automorphism $a \in A_0$ the permutation that it induces on $\Star(\one)$, namely
\begin{eqnarray*}
\overline{\cdot }: A_0 & \arrow & \Sym(\Sigma) \\
a & \mapsto & \overline{a}  \defeq c \circ a \circ (c|_{\Star(\one)})^{-1}
\end{eqnarray*}

\subsection{Color preserving automorphisms}
Given any vertex $v \in T$, denote by $(T,v)$ a copy of the tree $T$, but considered as a rooted tree with a base vertex at $v$.  In particular $S = (T,\one)$ is a specific model copy of this rooted tree. This tree has $d(d-1)^l$ vertices in the $l^{th}$ level and it's automorphism group is isomorphic to the compact group $A_0 = \Aut(T,\one)$. 

Given any two vertices $v,u \in T$ there is a unique automorphism $a_{(u,v)} \in A$ satisfying the following conditions:
\begin{itemize}
\item $u^{a_{(u,v)}} = v$,
\item $a_{(u,v)}$ preserves the colors of ``outgoing edges'', namely
$$c e^{a_{(u,v)}} = ce \quad \forall e \in ET {\text{ such that }} d(u,te) > d(u,oe).$$
\end{itemize}
The elements $a_{(u,v)}$ do not form a group, but they satisfy the condition 
\begin{equation} \label{eqn:cocycle1}
a_{(u,v)} a_{(v,w)} = a_{(u,w)}
\end{equation}
and in particular $a_{(v,u)} = a_{(u,v)}^{-1}$.

\subsection{Shadows} For every edge $e \in E^{+}T$ set $\Shadow[e] \defeq \{v \in T \ | \ te \in [oe,v] \}$. We will think of $\Shadow[e]$ as a $(d-1)$-ary rooted tree, with its base vertex at $te$. If $e, f \in E^{+}T$ are any two positively oriented edges, then 
$$\Shadow[e]^{a_{(te,tf)}} = \Shadow[f].$$
This useful property is due to the fact the edge coloring is compatible with the orientation.

\subsection{The local permutation cocycle.}
\begin{defn}
The {\it permutation cocycle} is defined to be the map
\begin{eqnarray*}
 \xi: & A \times T & \rightarrow A_0 \\
      & (a,v)  & \mapsto  a_{(\one,v)} \circ a \circ a_{(v^a,\one)}.
\end{eqnarray*}
The {\it local permutation cocycle} is defined as the composition of the permutation cocycle with the first congruence map. Explicitly this map assumes the form
\begin{eqnarray*}
\oxi: & A \times T & \rightarrow   \Sym(\Sigma) \\
      & (a,v)    & \mapsto  (c|_{\Star v})^{-1} \circ a|_{\Star v} \circ c.
\end{eqnarray*}
\end{defn}

\begin{rem} [The cocycle identity]
It is follows directly from equation (\ref{eqn:cocycle1}) that both $\xi$ and $\oxi$ satisfy the cocycle identity 
 identity:
 \begin{eqnarray*} \label{eqn:cocycle2}
     \nonumber \xi(ab,v) & = & \xi(a,v) \xi(b,v^a) \\
     \oxi(ab,v) & = & \oxi(a,v) \oxi(b,v^a) 
 \end{eqnarray*}
\end{rem}

\subsection{The cocycle and Haar measure.} \label{sec:co_Haar}
Since the collection of elements $\{a_{(\one,v)}\}_{v \in T}$ form a complete set of coset representatives for the group $A_0 = A_{\one}$; there is a bijection,
\begin{eqnarray*} \label{eqn:Phi}
\nonumber \Phi: A & \arrow & T \times A_0 \\
a & \mapsto & \left( \one^a, \xi(a,\one) \right) .
\end{eqnarray*}
The inverse of $\Phi$ is given by $\Phi^{-1}(v,b) = b \circ a_{(\one,v)}$. A useful feature of this map is its compatibility with Haar measure, 
$$\Phi_*(\Haar_A) = ({\text{Counting measure on }} T) \times \Haar_{A_0}.$$

We can decompose further, encoding the information about an element of $A_0$ according to its local permutation cocycle. 
\begin{eqnarray*}
\Psi: A_0  & \arrow & \Sym(\Sigma) \times \prod_{\one \ne v \in T}\Sym(\Sigma)_0 \\
b & \mapsto & \prod_{v \in T} \oxi(b,v)
\end{eqnarray*}
In the above equation, the requirement that the orientation $E^+T$ be preserved by $A_0$ is encoded by the fact that the local permutation cocycle fixes the color zero at all but the base vertex. This feature, which will greatly simplify our notation, is again due to our choice of a legal coloring that  is compatible with the orientation. 

This map too is compatible with Haar measure in the sense that 
\begin{equation*} \label{eqn:Psi} 
\Psi_*(\Haar_{A_0}) = \Haar_{\Sym(\Sigma)} \times \prod_{\one \ne v \in T} \Haar_{\Sym(\Sigma)_0}.
\end{equation*}
\noindent Combining the Equations (\ref{eqn:Phi}) and (\ref{eqn:Psi}) above we obtain
\begin{prop} \label{prop:Haar}
The map
\begin{eqnarray*}
A  & \arrow & T \times \Sym(\Sigma) \times \prod_{\one \ne v \in T}\Sym(\Sigma)_0 \\
b & \mapsto & \left(\one^b , \prod_{v \in T} \oxi(b,v)\right)
\end{eqnarray*}
is a measure preserving continuous bijection. Where the measure on the right is the product of the counting measure on $T$ and the product measure on the compact group.
\end{prop}
The fact that all the above maps are measure preserving can be directly verified, using the cocycle identity, by checking that the measures on the right are invariant. 

\subsection{Subtrees}
An analysis similar to the one carried out in the last section holds also for subtrees.  We will demonstrate this in a specific case that would be of interest for us, namely the shadow of a positively oriented edge.

Let $e \in E^+T$ be a positive edge, $Y = \Shadow[e]$, its shadow considered as a rooted tree. The automorphism group $\Aut(Y)$ is naturally a subquotient of $A$, namely there is a short exact sequence:
$$1 \arrow A_{Y} \arrow A_{\{Y\}} \arrow \Aut(Y) \arrow 1,$$ 
where $A_{\{Y\}}$ stands for the setwise stabilizer and $A_{Y}$ stands for the pointwise stabilizer of the tree $Y$. This short exact sequence splits $\iota: \Aut(Y) \hookrightarrow A$, so that $\Aut(Y)$ can also be realized as a subgroup of $A$.  
\begin{equation} \label{eqn:iota}
\iota(a)(x) = \left\{ \begin{array}{ll} a(x) & x \in \Shadow[e] \\ x & x \not \in \Shadow[e] \end{array} \right.
\end{equation}

The tree $Y$ inherits a legal coloring from the ambient tree $T$. Since $Y$ is the shadow of a positive edge, the orientation inherited from $T$ coincides with the natural orientation on $Y$. We may therefore ignore all negative edges, and consider the coloring as a map $c : E^{+}Y \arrow \{1,2, \ldots d-1\}$. The local permutation cocycle is defined in exactly the same way as it is defined for the ambient tree $\oxi: Aut(Y) \times Y \arrow \Sym(\Sigma)_0$, where we have identified here $\Sym(\Sigma \setminus \{0\}) \cong \Sym(\Sigma)_0$. Just like in the ambient tree, the local permutation cocycle completely characterizes the automorphism and is compatible with Haar measure
\begin{eqnarray*}
\Aut(Y) & \stackrel{\arrow}{~} & \prod_{x \in Y} \Sym(\Sigma)_0 \\
a & \mapsto & \prod_{x \in Y} \oxi(a,x) \\
\Haar_Y & \mapsto & \prod_{x \in Y} \Haar_{\Sym(\Sigma)_0}
\end{eqnarray*}
The use the same notation for the cocycle on both trees is justified by the following claim whose verification is trivial.
\begin{claim}
Let $a,b \in A_{\{Y\}}$ be two elements such that $a|_Y = b|_Y$. Then 
$$\oxi^T(a,x) = \oxi^T(b,x) = \oxi^Y(a|_Y,x), \quad \forall x \in Y.$$
where $\oxi^T, \oxi^Y$ represent the local permutation cocycle with respect to the trees $T$ and $Y$ respectively. 
\end{claim}

The following lemma, that follows directly form the above observations, will be very useful for us in the proof of the main theorem.
\begin{lem} \label{lem:Haar_LPC}
Let $Y \subset T$ be the shadow of a positive edge as above, $a \in A_{\{Y\}}$ a (not necessarily Haar) random group element. Then $a|_Y$ admits the distribution of a $\Haar_Y$-random element, if and only if 
$$\{\oxi (a,x) \ | \ x \in Y\}$$ are mutually independent, $\Haar_{\Sym(\Sigma)_0}$-random elements of $\Sym(\Sigma)_0$. 
\end{lem} \qed

\section{Proof of the main theorem} \label{sec:proof}
We suggest that the readers refer to Figure \ref{fig:genII} for an illustration of some of the geometric ideas.
\subsection{Fixing a word $\mathbf{w \in \Gamma * \Z}$}
Assume that $\Gamma < A$ is a countable group that acts almost freely on the tree. Consider the free product $\Gamma * \Z$, denoting the generator of the cyclic group by $t$. For $a \in A$ there is a canonical {\it evaluation} homomorphism which is the identity on $\Gamma$ and sends $t$ to $a$.
\begin{eqnarray*}
  \Phi_a: \Gamma * \Z & \arrow & \langle \Gamma , a \rangle\\
  w & \mapsto & w(a).
\end{eqnarray*}
As $\Gamma * \Z$ is countable {\it our main theorem will follow} by showing that for every fixed $1 \ne w \in \Gamma * \Z$ and for almost every $a \in A$; the element $w(a) \defeq \Phi_a(w)$ fixes but finitely many vertices. We will henceforth fix such an element $\id \ne w \in \Gamma * \Z$.

\subsection{Conditioning that $\mathbf{w(a) \one = \one}$} If $w(a)$ is hyperbolic or an inversion, then it fixes no vertex on the tree, therefore we may assume that $w(a)$ is elliptic. Since Haar measure is invariant under conjugation, we may assume, without loss of generality, that the fixed vertex is $\one$. Therefore, setting $\Omega = \{a \in A \ | \ w(a) \in A_0\}$,  {\it it would suffice} to prove that
\begin{equation} \label{eqn:one_word}
 \Haar_A (\{ a \in \Omega \ | \ w(a) {\text{ has infinitely many fixed points}} \}) = 0.
 \end{equation}

\subsection{A canonical form for $\mathbf{w}$} Elements of $\Gamma * \Z$ admit a unique canonical normal form. For our fixed element $w$ this takes the form
 $$w = w_0 w_1 w_2 w_3 \ldots w_{n}.$$
Where,
\begin{itemize}
 \item each $w_i$ is either $t$ or $t^{-1}$ or an element of $\Gamma$,
 \item $w_i, w_{i+1}$ are never both elements of $\Gamma$,
 \item $w_i w_{i+1}$ are never of the form $t t^{-1}$ or $t^{-1}t$.
\end{itemize}
 
\begin{defn} A word $w$ is called {\it cyclicly reduced} if it is not a conjugate of another word whose canonical form is shorter.
\end{defn}
Since  the number of fixed vertices in the tree is a conjugacy invariant, it will be enough to prove Equation \ref{eqn:one_word} for cyclicly reduced words. We will therefore assume, without loss of generality, that our fixed word $w$ is cyclicly reduced.

\subsection{Traces} Given a vertex $y \in T$, define functions\footnote{One should think of these functions as ``random variables''. We insist on calling them functions just to emphasize the fact that they are defined on a measure space which possibly has infinite measure; reserving the term random variable only for functions defined on a probability space.} of the variable $a \in \Omega$ as follows $y_0(a) = y$, $y_1(a) = y_0(a)^{w_0(a)}, y_2(a) = y_1(a)^{w_1(a)}, \ldots, y_{n+1}(a) = y_{n}(a)^{w_{n}(a)} = y^{w(a)} \in T$. Similarly if $e \in ET$ is an edge, define functions $e_0(a) = e$, $e_1(a) = e_0(a) ^{w_0(a)}, \ldots, e_{n+1}(a) = e^{w(a)}$. All of these are functions of $a$ but we will usually suppress $a$ from the notation. The sequence of vertices $(y_0,y_1,y_2,\ldots,y_{n+1})$ is called the {\it trace of $y$ under $w$}. A trace of a vertex or of an edge $\mathcal{T} =\{l_0,l_1,\ldots,l_n,l_{n+1}\}$ is called  {\it simple} if all its elements are distinct, except for a possible equality $l_0 = l_{n+1}$. 

Denote by $C \defeq \Conv \{\one_0,\one_1,\ldots,\one_{n+1} = \one\}$, the convex hull of the trace of the base vertex. Note that $C$ is a function assigning to every $a \in \Omega$ a finite subtree $C(a) \subset T$.

\subsection{Main induction}
The claim will be proved by induction on $n$ - the length of the expansion of $w$ in canonical form. The point is that, since by induction all subwords of $w$ have but finitely many fixed points, then $w$ itself will admit but finitely many vertices whose trace is not simple. The origin of this geometric idea comes from the  proof of \cite[Corollary 4.2]{AV-dimension_theory}; but the geometry here is more complicated due to the existence of hyperbolic elements. 

\subsection{Basis of induction} Let $n=0$. When $\id \ne w = w_0$ is an element of $\Gamma$, it has only finitely many fixed points by assumption. Otherwise $w \in \{t,t^{-1}\}$. In this case $\Omega = A_0$ and $a$ is a $\Haar_{A_0}$-random element. The fact that $a$ almost surely has finitely many fixed points in this case is proved by Ab\'{e}rt and Vir\'{a}g. We quote their theorem here because we will use it a few times in the proof.
\begin{prop} (Ab\'{e}rt-Vir\'{a}g \cite[corollary 2.7]{AV-dimension_theory}) \label{prop:AV_basis}
Let $S$ be a spherically homogeneous locally finite rooted tree. Then a $\Haar_S$-random element $a \in \Aut(S)$ almost surely has but finitely many fixed points.
\end{prop} 
\begin{proof}
For a complete proof see the above reference. In short, what they show is that the tree of fixed points assumes the distribution of a critical Galton-Watson tree and therefore is almost surely finite.
\end{proof}

\subsection{The sphere of radius $\mathbf{M}$}
\begin{lem} \label{lem:M}
There exists a measurable function $M: \Omega \arrow \N \cup \{\infty\}$, which is finite almost everywhere, satisfying the following properties:
 \begin{enumerate}
 \item For every $y \in T$ with $d(y,\one) \ge M(a)$ the trace $(y_0, y_1,\ldots y_{n+1})$ is simple. 
 \item For every edge $e \in E^{+}T$ with $d(\one,e) \ge M$ all the edges in the trace $(e_0,e_1,e_2 \ldots, e_{n+1}) \subset E^{+}T$, are positively oriented.
 \end{enumerate}
\end{lem}

\begin{proof}
The first statement follows directly from the induction hypothesis: choose $M$ large enough so that the ball $B_{M}(\one)$ contains all the fixed vertices for all the (cyclic) subwords of $w$, of the form $w_i w_{i+1} \ldots w_{i+k \ [\hspace{-5pt}\mod n+1]}$ with $0 \le k < n$. Note that, by the assumption that $w$ is cyclicly reduced, all the cyclic subwords of $w$ are still in canonical form and the induction hypothesis does hold for them. 

It is automatically true that $[\one_i,t e_i]^{w_i} =[\one_{i+1},te_{i+1}]$ and if $M$ is large enough that $M = d(\one,te) = d(\one_i,te_i) > \diam C+1$ then $e_i$ which is the last edge of $[\one_i,te_i]$ will have to face away from the convex set $C$ and hence away from $\one$.
\end{proof}

\begin{rem} \label{rem:ordering_trees}
Let $e$ be an edge such that $d(\one,e) > M$. It follows directly from Lemma \ref{lem:M} above that for $0 \le i \ne j \le n$ the corresponding shadows satisfy:
\begin{itemize}
\item Either $\Shadow[e_i] \cap \Shadow[e_j] = \emptyset$. 
\item Or $\Shadow[e_i] \subsetneqq \Shadow[e_j]$, up to possibly exchanging $i$ and $j$.
\end{itemize} 
\end{rem}

Let $S = S(a) = \{e \in E^{+}T \ | \ d(\one,oe) = M\}$ be the {{\it sphere of edges of radius $M$ around $\one$}}, where $M = M(a)$ is the number appearing in Lemma \ref{lem:M}. Since
 $$T \setminus \left( \cup_{e \in S} \Shadow[e] \right) = B_M(\one),$$
is a finite set, {\it the proof will be concluded by} establishing that  for almost every $a \in \Omega$ and for every $e \in S$ the automorphism $w = w(a)$ has only finitely many fixed points inside $\Shadow[e]$. 

Let us fix such an edge $e \in S$. The easy case is when $e \ne e^w$. In this case there are two options, either $\Shadow[e]$ and $\Shadow[e^w] = \Shadow[e]^w$ are completely disjoint or one of them properly contains the other. In any case there are no fixed vertices in $\Shadow[e]$. In what follows it will be assumed that our fixed edge $e$ has a closed trace, namely that $e^{w} = e$, as in Figure \ref{fig:genII}.

\subsection{Conditioning on the trace.}
Let us define 
$$\Omega_{\T} = \{a \in \Omega \ | \ (e_0,e_1,e_2,\ldots e_{n+1} = e_n) = \T \},$$
 where $\T$ is some fixed trace. Since there are only countably many possibilities for the trace of $e$, we can prove the proposition for each trace $\T$ separately. The advantage of this approach is that $\Omega_{\T}$ has finite Haar measure. Let us fix a trace $\T$ and re-normalize Haar measure so that $\Omega_{\T}$ becomes a probability space. From now on, for simplicity of notation, denote $\Shadow[i] \defeq \Shadow[e_i]$. Note that since the trace $\T$ is simple these, shadows will satisfy the conditions of Remark \ref{rem:ordering_trees}, namely they will be either disjoint, or properly contained in each other. 

\subsection{A special edge in the trace}
From all the edges $e_i \in \T$ in our trace we will fix one $e_I; \ 0 \le I \le n$ that satisfies the following conditions
\begin{itemize}
\item Either $w_I =t$ or $w_{I-1} = t^{-1}$,
\item $\Shadow[I]$ is deeper into the tree than any other edge in the trace satisfying the above condition. Namely if $0 \le i \le n; \  i \ne I$ is another index such that $w_i =t$ or such that $w_{i-1} =t^{-1}$ then either $\Shadow[i] \cap \Shadow[I] = \emptyset$ or $\Shadow[I] \subsetneqq \Shadow[i]$.
\end{itemize}
In Figure \ref{fig:genII}  for example, if $w_2 = t$ or if $w_1 = t^{-1}$ it is possible to choose $I = 2$, but it would be illegal to choose $I = 0$ because $\Shadow[2] \varsubsetneqq \Shadow[0]$.
\begin{figure}
\begin{center}
\includegraphics[width=10cm]{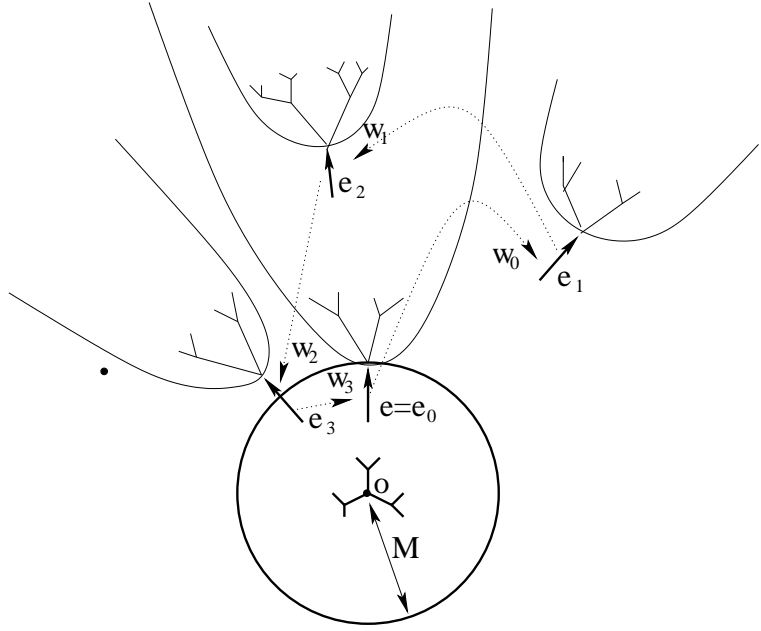}
\caption{A typical arrangement of the shadow trees of the edges in $\T$.}
\label{fig:genII}
\end{center}
\end{figure}

\subsection{Looking for a Haar random element in the rooted tree.}
Since $e^w = e$, $w(a)$ fixes the $(d-1)$-regular rooted tree $Y \defeq \Shadow[e]$,  $w(a)|_{Y}$ is a random automorphism of $Y$. If we can show that $w(a)|_{Y}$ admits the distribution of a $\Haar_Y$-uniform element, Proposition \ref{prop:AV_basis} will imply that $Y$ almost surely contains but finitely many fixed points, and {\it our theorem will follow}. 

By Lemma \ref{lem:Haar_LPC} $w|_Y$ is $\Haar_Y$-random if the random variables $$\{\oxi(w,x) \in \Sym(\Sigma)_{0} \ | \ x \in Y \}$$ are mutually independent and uniformly distributed. To check this latter condition, it is enough to order the vertices in $Y$ and verify that each one is independent of the joint distribution of all the previous ones$$Y = \{te = x^0,x^1,x^2,x^3, \ldots\}.$$ 
We will require that this ordering be consistent with the distance from the root; i.e. that $d(te,x^j) \le d(te,x^{j+1}) \ \forall j$. In particular this implies that $x^0 = te$. 

\begin{rem} \label{rem:reduction}
To summarize all the reductions so far, {\it the main theorem will be proved} if we show that $\oxi(w(a),x^j)$ is a uniform element of $\Sym(\Sigma)_{0}$ and independent of the joint distribution of $\{\oxi(w(a),x^k) \ | \ 0 \le k < j \}$
\end{rem} 

\subsection{Decomposing the cocycle}
Let us expand $\oxi(w,x^j)$ using the cocycle condition
\begin{equation} \label{eqn:expansion}
\oxi (w,x^j) = \oxi(w_0,x^j_0) \oxi(w_1,x^j_1) \ldots  \oxi(w_n,x^j_n).
\end{equation}
where $(x^j_0 ,x^j_1,\ldots, x^j_n)$ is the trace of $x^j$. 

\subsection{Singling out certain values of the cocycle} \label{sec:sing}
Recall the two possibilities occurring in the definition of the special index $I$, either $w_I = t$ or $w_{I-1}^{-1} = t$. Let us single out certain special values of the cocycle, the definition of which will vary slightly according to these two possibilities
$$
\OXi^j  \defeq   
   \left \{ \begin{array}{ll} 
      \oxi(w_I, x_I^j) = \oxi(a,x^j_I)  & {\text{ if }} w_I = t \\
      \oxi(w_{I-1},x^j_{I-1}) = \oxi(a^{-1},x^j_{I-1})  & {\text{ if }} w_{I-1} = t^{-1}
   \end{array} \right. .
$$
We claim that, as $a$ ranges over the probability space $\Omega_{\T}$,
\begin{enumerate}
\item \label{itm:uniform} $\OXi^j$ are uniform elements of $\Sym(\Sigma)_0$. 
\item \label{itm:ind} $\OXi^j$ is independent of the mutual distribution of the other random variables in the expansion \ref{eqn:expansion} 
  $$\{\oxi(w_i,x^k_i) \ | \ 0 \le i \le n, \ k \le j\} \setminus \{\OXi^j\}.$$§ñ
\end{enumerate}

\subsection{The statistical distribution of the cocycle values.}
\begin{defn} 
Let $\Xi: \Omega_{\T} \arrow \Upsilon$ be a random variable into a probability space $\Upsilon$, H a group and  $H \curvearrowright^{\times} \Upsilon$ a measure preserving action. We say that $H$ {\it{acts on the random variable $\Xi$ via the given action}} if there exists a measure preserving action $H \curvearrowright^{*} \Omega_{\T}$ such that the two actions are intertwined by $\Xi$, in the sense that $\Xi(\sigma * a) = \sigma \times \Xi(a),\ \forall \sigma \in H, \ a \in \Omega_{\T}$. Such an action will be denoted by $H \curvearrowright^{\times} \Xi.$
\end{defn}

Claims (\ref{itm:uniform}),(\ref{itm:ind}) of Section \ref{sec:sing} above concerning the statistics of the random variables $\{\oxi(w_i,x^k_i) \ | \ 0 \le i \le n, \ k \le j\}$, will be proved using certain actions $\Sym(\Sigma)_0 \curvearrowright^{\stackrel{j}{\times}} \oxi(w_i,x^k_i)$ of the finite group $\Sym(\Sigma)_0$ on these random variables. The action $\stackrel{j}{\times}$ will be transitive on $\OXi^j$ while fixing the other random variables. In fact the former action will be the left regular action of $\Sym(\Sigma)_0$ on itself in the case where $w_I = t$ and the right regular action of the same group in case $w_{I-1} = t^{-1}$.

The existence of such actions will be established in the following two sections. For now, assume that such actions have already been constructed and let us demonstrate how properties (1), (2) above follow. Assume that $w_I = t$, the treatment of the other case $w_{I-1} = t^{-1}$ being almost identical. For the first property write:
$$P \left\{ \OXi^j = \sigma \right\}
   =  P \left\{ \OXi^j( \sigma' \stackrel{j}{*} a) = \sigma' \sigma \right\} 
   =  P \left\{ \OXi^j = \sigma' \stackrel{j}{\times} \sigma \right\}$$
This concludes the proof of (1) because, since the action is transitive,  $\sigma'  \stackrel{j}{\times} \sigma$ is an arbitrary element of $\Sym(\Sigma)_0$. For the second property, let $f(a)$ be any random variable, which depends on $a$ only as a function of the values of the random variables $\{\oxi(w_i,x^k_i) \ | \ 0 \le i \le n, 0 \le k \le j \} \setminus \{\OXi^j\}$ (excluding $\OXi^j$). Then for every measurable set $B$ and every $\sigma' \in \Sym(\Sigma)_0$
\begin{eqnarray*}
P \left\{ (f \in B) {\text{ and }} (\OXi^j = \sigma) \right\} 
   & = & P \left\{ (f(\sigma' \stackrel{j}{*} a) \in B) {\text{ and }} (\OXi^j(\sigma' \stackrel{j}{*} a)) = \sigma \right\} \\
   & = & P \left\{ (f \in B) {\text{ and }} (\OXi^j = (\sigma')^{-1} \stackrel{j}{\times} \sigma) \right\}
\end{eqnarray*} 
which gives the desired independence result, again because $(\sigma')^{-1} \stackrel{j}{\times} \sigma$ is an arbitrary element of $\Sym(\Sigma)_0$.

\subsection{Construction of the actions $\mathbf{Sym(\Sigma)_0 \curvearrowright ^{\stackrel{j}{\times}} \oxi(w_i,x^k_i)}$}

Given an index $j \in \N$ consider the embedding 
$\eta^j: \Sym(\Sigma)_0 \arrow \Aut(T).$
Where the automorphism $\eta^j \sigma \in \Aut(T)$ is defined,  via the identification described by Proposition \ref{prop:Haar} 
\begin{equation*}
\eta^j \sigma \one = \one {\text{ and,}} \hspace{1.5cm}
\oxi(\eta^j \sigma ,x)  = \left\{ \begin{array}{ll} 
  \sigma & {\text{ if }} x = x^j_I  \\
  \id & {\text{ otherwise}} \end{array} \right. .
\end{equation*}
In particular, setting $f^j$ to be the unique positively oriented edge such that $tf^j = x^j_I$ and $Z := \Shadow[f^j]$. The automorphism $\eta^j \sigma$ fixes $T \setminus Z$ pointwise. This embedding, composed with the left regular action of $\Aut(T)$ on itself, gives rise to a measure preserving action $\Sym(\Sigma)_0 \curvearrowright^{\stackrel{j}{*}} \Omega_{\T}$
$$
 \sigma *^{j} a \defeq (\eta^j \sigma) a.
$$

To verify this we only need to show that the subset $\Omega_{\T}$ is invariant. Given $a \in \Omega_{\T}, \ \sigma \in \Sym(\Sigma)_0$, let us verify that $\sigma\stackrel{j}{*}a = (\eta^j \sigma) a \in \Omega_{\T}$. Explicitly this means that $\T = \T'$ where 
$$\T' \defeq \operatorname{Trace}(\sigma^j a)  = \left \{e = e'_0, e'_1 = (e'_0)^{w_0(\sigma \stackrel{j}{*} a)}, e'_2 = (e'_1)^{w_1(\sigma \stackrel{j}{*} a)} \ldots \right\}.$$ By induction on $i$ let us assume that $e_i = e'_i$ and show that $e'_{i+1} = e_{i+1}$. If $w_i \in \Gamma$ the induction step is obvious because $w_i$ is a constant independent of its argument. If $w_i =t$ then, by choice of the index $I$, either $i = I$ and $e'_i = e_i = e_I$ or $e_i \not \in Z$. In both cases $e_i^{\eta^j \sigma} = e_i$ so that $e'_{i+1} = e_i^{(\eta^j \sigma) a} = e_i^{a} = e_{i+1}$. Finally if $w_i = t^{-1}$ then, by choice of the index $I$, either $i+1 = I$ and $e_{i+1} = e_I$ or $e_{i+1} \not \in Z$. In any case $e'_{i+1} = e_i^{a^{-1} (\eta^j \sigma)^{-1}} = e_{i+1}^{(\eta^j \sigma)^{-1}} = e_{i+1}$. Hence by induction $e'_i = e_i, \ \forall 0 \le i \le n$ so $\T' = \T$ and $\sigma \stackrel{j}{*} a \in \Omega_{\T}$. 

\subsection{The effect of the actions on the random variables.} Fixing an index $j \in \N$, we will now show that the action $\stackrel{j}{*}$ has the following effect on the random variables $\oxi(w_i,x^k_i)$:             
\begin{equation*}
\oxi\left( w_i \left( \sigma \stackrel{j}{*} a \right) , x^k_i \right)  = \left\{ \begin{array}{ll} 
  \sigma \oxi(w_i(a),x^k_i) & {\text{ if }} k=j ,  i = I  {\text{ and }} w_I = t  \\
  \oxi(w_i(a),x^k_i) \sigma^{-1} & {\text{ if }} k =j , i = I-1  {\text{ and }} w_{I-1}^{-1} = t\\
  \oxi(w_i(a),x^k_i) & {\text{ otherwise}}  
 \end{array} \right. 
\end{equation*}
Let us denote by $D = \{x^k_i \ | \ 0 \le k \le j, \ 0 \le i \le n, \ w^k_i = t {\text{ or }} w^k_{i-1} = t^{-1}\}$. Substituting into $w_i$ the appropriate values $t, t^{-1}$ or a letter of $\Gamma$ whenever necessary; we see that in order to verify the above statement it is enough check that ``for every $x^k_i \in D$ such that $\oxi(\sigma \stackrel{j}{*} a,x^k_i) \ne \oxi(a,x^k_i)$ we actually have $I = i$ and $k = j$''.  Such a vertex $x^k_i \in D$  must be contained in $\Shadow[f^j] \subset \Shadow[I]$, since $\eta^j \sigma$ fixes  $T \setminus \Shadow[f^j]$ pointwise. Since $x^k_i \in \Shadow[I]$ then $\x^k_i = x^{k'}_I$ for some $k' \in \N$. Furthermore $k' \ge j$ because $x^k_i = x^{k'}_I \in \Shadow[f^j]$ (with $tf^j = x^j_I$) and the ordering of the vertices is compatible with the distance from the root.  

Let us first demonstrate that $i = I$. If this were not the case then $e_i \ne e_I$ because, by the induction hypothesis, the trace of $e$ is simple. These shadows cannot be disjoint, because $x^k_i = x^{k'}_I \in \Shadow[i] \cap \Shadow[I] \ne \emptyset$ so it follows from the definition of the special index $I$ that $\Shadow[I] \subsetneqq \Shadow[i]$ and hence
$$\level_{i}(x^k_i) > \level_{I}(x^k_i) = \level_{I}(x^{k'}_I) = \level_{i}(x^{k'}_i).$$
Here $\level_i$ denotes the level with respect to the rooted tree $\Shadow[i]$. But, because $k < j \le k'$, this is a contradiction to the fact that the ordering of the vertices $x^0_i,x^1_i,x^2_i \ldots$ in the rooted tree $\Shadow[i]$ is compatible with the depth. Therefore $i=I$ and using this it follows that $x^k_i = x^{k'}_I = x^{k'}_i$ which implies that $k = k'$. But $k \le j \le k'$ so that $j = k$ as desired. 

\begin{rem} Note that the argument above shows also that the vertex $x^j_I$ appears in $D$ exactly once.
\end{rem}
\subsection{The Haar measure trick.} Consider the cocycle equation \ref{eqn:expansion}, and write it in the form:
$$\oxi (w,x^j) = A^j \OXi^j B^j,$$
where $A^j,B^j$ are exactly what they should be in order to make this work:
\begin{eqnarray*}
A^j & \defeq & \left \{ \begin{array}{ll}
   \oxi(w_0,x^j_0) \oxi(w_1,x^j_1) \ldots \oxi(w_{I-1},x^j_{I-1}), & {\text{ if }} w_I = t \\
   \oxi(w_0,x^j_0) \oxi(w_1,x^j_1) \ldots \oxi(w_{I-2},x^j_{I-2}), & {\text{ if }} w_{I-1} = t^{-1} 
   \end{array} \right. \\
B^j & \defeq &   \left \{ \begin{array}{ll}
   \oxi(w_{I+1},x^j_{I+1}) \oxi(w_{I+2},x^j_{I+2}) \ldots \oxi(w_n,x^j_n), & {\text{ if }} w_I = t \\
   \oxi(w_{I},x^j_{I}) \oxi(w_{I+1},x^j_{I+1}) \ldots \oxi(w_n,x^j_n), & {\text{ if }} w_{I-1} = t^{-1} 
   \end{array} \right. 
\end{eqnarray*}

Summarizing the information about the statistical distribution of the random variables $\oxi(w_i,x^k_i)$ we conclude that 
\begin{enumerate}
 \item $\OXi^j$ are $\Haar_{\Sym(\Sigma)_0}$-random elements. 
 \item $\OXi^j$ is independent of the mutual distribution of all the random variables 
 $\{A^k,\OXi^k,B^k \ | \ 0 \le k \le j\} {\text{ excluding }} \OXi^j.$
\end{enumerate}
These properties are inherited by the random variables $\oxi(w,x^j)$.
\begin{enumerate}
 \item $\oxi(w,x^j)$ is a $\Haar_{\Sym(\Sigma)_0}$-random element. 
 \item $\oxi(w,x^j)$ is independent of the mutual distribution of all the random variables 
 $\{\oxi(w,x^k) \ | \ 0 \le k \le j\}.$
\end{enumerate}
Indeed if we condition on the values of the random variables $\{A^k, B^k \ | \ 0 \le k \le j\}$, then the two properties above follow directly from the corresponding properties for the factors. In particular, the uniform distribution for $\oxi(w,x^j)$ follows from the same property for $\OXi^j$ by the invariance of $\Haar_{\Sym(\Sigma)_0}$-measure under left and right multiplication. Now, integrating over all possible values of $A^k,B^k$, the desired properties of the random variables $\oxi(w,x^j)$ follow unconditionally. 

This completes the proof of the theorem, as observed in Remark \ref{rem:reduction}. Indeed it now follows that 
$$w|_{\Shadow[e]} \in \Aut(\Shadow[e])$$ 
admits the distribution of a $\Haar_{\Aut(\Shadow[e])}$-random element because the random variables $\oxi(w,x^j)$ form the local permutation data corresponding to this rooted tree automorphism. Hence, by Ab\'{e}rt and Vir\'{a}g's Proposition \ref{prop:AV_basis}, the element $w|_{\Shadow[e]}$ has but finitely many fixed vertices in $\Shadow[e]$. But this is true for every $w$-fixed edge $e \in S$ in the sphere of radius $M$; so that $w$ admits only finitely many fixed vertices in the whole tree. 

\bibliographystyle{amsalpha}
\bibliography{../tex_utils/yair}
\end{document}